\documentclass[a4paper,twoside,11pt]{article}

\usepackage{a4wide}
\usepackage{amssymb}
\usepackage{graphicx}
\usepackage{url}

\usepackage{amsmath}
\usepackage{amsthm}
\usepackage{algorithm}
\usepackage{algpseudocode}

\usepackage{fancyhdr}
\usepackage{authblk}

\newtheorem{theorem}{Theorem}

\begin{document}

\title{Regularization with Sparse Vector Fields: \\From Image Compression to TV-type Reconstruction}

\author[*]{Eva-Maria Brinkmann}
\author[*]{Martin Burger}
\author[+,*]{Joana Grah}
\affil[*]{Institute for Computational and Applied Mathematics, Westf\"alische Wilhelms-Universit\"at M\"unster, Germany}
\affil[+]{Department of Applied Mathematics and Theoretical Physics, University of Cambridge, UK}

\maketitle

\begin{abstract}
This paper introduces a novel variational approach for image compression motivated by recent PDE-based approaches combining edge detection and Laplacian inpainting. The essential feature is to encode the image via a sparse vector field, ideally concentrating on a set of measure zero. An equivalent reformulation of the compression approach leads to a variational model resembling the ROF-model for image denoising, hence we further study the properties of the effective regularization functional introduced by the novel approach and discuss similarities to TV and TGV functionals. Moreover we computationally investigate the behaviour of the model with sparse vector fields for compression in particular for high resolution images and give an outlook towards denoising.\\
\\
\textbf{Keywords:} Image compression, denoising, reconstruction, diffusion inpainting, sparsity, total variation
\end{abstract}

\section{Introduction}

Image compression is a topic of interest since the beginning of digital imaging, remaining relevant continuously due to the ongoing improvement in image resolution. While standard approaches are based on orthogonal bases and frames like cosine transforms or wavelets, an alternative route based on ideas from partial differential equations has emerged recently (cf. \cite{Mainberger_et_al,Mainberger_and_Weickert}). In the latter case particular attention is paid to compressions from which cartoons can be reconstructed accurately avoiding the artefacts of the above mentioned standard approaches by a direct treatment of edges. Roughly speaking, their idea is to detect edges and store the image value in pixels on both sides of an edge. The remaining parts of the image are completed by harmonic inpainting. An alternative interpretation of the two-sided pixel values, worked out more clearly in the osmosis setting of \cite{Weickert_et_al}, is that a vector field $v$ corresponding to the normal derivatives of the image at the edge location is stored. The inpainting of the image $u$ then corresponds to a solution of 
\begin{equation}
	\Delta u = \nabla \cdot v \qquad \mbox{in~}\Omega\,, \label{eq:inpainting}
\end{equation}
where $\Omega$ is the image domain into which $v$ is extended by zero off the edges. 

In this paper we start from reinterpretation of the PDE-based compression in terms of sparsity, which directly translates into a variational framework. The key observation is that with the edge detection and zero extension of $v$ one essentially looks for a sparse vector field that leads to a certain precision in the reconstruction via (\ref{eq:inpainting}). In the spirit of the predominant sparsity regularization of imaging problems we study a direct variational approach: we minimize an $L^1$-type norm of the vector field subject to the constraint that $u$ reconstructed via (\ref{eq:inpainting}) approximates a given image $f$ up to a certain tolerance, i.e.
\begin{equation}
	\Vert v \Vert_1 \rightarrow \min_{u,v} \qquad\mbox{subject to}\quad \Delta u = \nabla \cdot v\,, \quad \Vert u - f \Vert_2 \leq \epsilon\,. \label{eq:model0}
\end{equation}
As we shall discuss below a more rigorous statement of the problem takes into account that $v$ needs to be interpreted as a vectorial Radon measure on $\Omega$ (similar to gradients of $BV$-functions) in a continuum setting. The properties of a limiting continuum model appear to be of particular advantage for the compression issues, since one expects to concentrate the measure $v$ on a set of Lebesgue measure zero. This means that for a suitable discrete approximation of the model the number of pixels we need to store the vector field in divided by the total number of pixels tends to zero. A direct consequence is an increase in the compression rates as the image resolution increases, a highly desirable property. 

Apart from compression we shall investigate models like the above one as a regularization for more general imaging problems. The relation to denoising models can readily be seen when the above constraint of approximating $f$ is incorporated via a Lagrange functional. Indeed, there exists a Lagrange parameter $\lambda > 0$ such that (\ref{eq:model0}) is equivalent to 
\begin{equation}
	\frac{\lambda}2 \Vert u - f \Vert_2^2 + \Vert v \Vert_1 \rightarrow \min_{u,v} \qquad\mbox{subject to}\quad \Delta u = \nabla \cdot v\,. \label{eq:model1}
\end{equation}
Hence, we may also interpret the norm of $v$ as an implicit regularization of $v$ and simply replace the first term by an arbitrary data term to treat other imaging problems. This is more apparent when we replace the constraint by a natural special solution $v = \nabla u$. In this case (\ref{eq:model1}) is just the ROF-model for denoising (cf. \cite{Rudin_et_al}) and indeed an equivalence relation holds in spatial dimension one. In higher spatial dimensions it becomes apparent that a key role is played by the curl $\nabla \times v$. If the curl of $v$ vanishes, it simply becomes a gradient vector field and hence we again recover the ROF-model. This motivates to study further generalizations of the model also penalizing $\nabla \times v$. With the interpretation that $v$ is related to the gradient of $u$, the additional term becomes a higher order regularization effectively. Functionals of this type are currently studied in particular to reduce staircasing artefacts in total variation regularization (cf. \cite{Chan_et_al,Bredies_et_al,Benning_highOrderTV_2013}), and indeed we shall be able to draw very close analogies to the recently very popular TGV-approach (cf. \cite{Bredies_et_al}). The reduction of staircasing is confirmed by computational experiments. However, in the denoising case we shall see that the divergence part alone does not suffice for appropriate smoothing.

The remainder of the paper is organized as follows: In Section 2 we discuss the model and its variants. Then, we proceed to a discussion of some theoretical properties in Section 3, which provide some insights into the sparse vector field model. Section 4 introduces a numerical solution based on primal-dual optimization methods, which is used for some experimental studies in Section 5. Finally, we provide a conclusion and directions for future research.

\section{Variational Model and Regularization Functionals}
\label{chap:Variational_Model}

In order to obtain an appropriate formulation of (\ref{eq:model1}) we proceed as in \cite{Bredies_Pikkarainen} and interpret $v$ as a $d$-dimensional Radon measure on $\Omega \subset \mathbb{R}^d$. The regularization functional is then 
\begin{equation}
 \Vert v \Vert_{{\cal M}(\Omega)} = \sup_{\varphi \in C(\Omega)^d, \Vert \varphi\Vert_\infty \leq 1} \int_\Omega \varphi \cdot ~dv\,,
\end{equation}
where (\ref{eq:inpainting}) is to be understood in a weak form as well. Hence, (\ref{eq:model1}) is rewritten as
\begin{equation}
	\frac{\lambda}2 \Vert u - f \Vert_2^2 +  \Vert v \Vert_{{\cal M}(\Omega)} \rightarrow \min_{u \in L^2(\Omega),v \in {\cal M}(\Omega)^d} \qquad\mbox{subject to}\quad \Delta u = \nabla \cdot v\,. \label{eq:model1a}
\end{equation}

The model can be formulated as in recent approaches for denoising by defining $w = \nabla u - v $ (in the sense of distributions). With $\chi_0$ being the characteristic function of the set $\{0\}$ we obtain
\begin{equation}
	\frac{\lambda}2 \Vert u - f \Vert_2^2 +  \Vert \nabla u - w \Vert_{{\cal M}(\Omega)}
	+ \chi_0(\nabla \cdot w) \rightarrow \min_{u \in L^2(\Omega),w \in {\cal D}'(\Omega)^d}\,. \label{eq:model1b}
\end{equation}
One observes that the regularization functional is now an infimal convolution of total variation and a functional of $\nabla w$. The same structure is apparent in the recently popularized TGV-model (cf. \cite{Bredies_et_al}), which in the analogous setting reads
\begin{equation}
	\frac{\lambda}2 \Vert u - f \Vert_2^2 +  \Vert \nabla u - w \Vert_{{\cal M}(\Omega)}
	+ \Vert \nabla w \Vert_{{\cal M}(\Omega)} \rightarrow \min_{u \in L^2(\Omega),w \in {\cal M}(\Omega)^d}\,. \label{eq:tgv}
\end{equation}
A major difference of the TGV approach to our new model is the fact that in our case only the divergence of $v$ respectively $w$ is penalized, which might be too weak to achieve suitable regularization properties. This can obviously be realized if we add additional regularization terms depending on $\nabla \times v$, which is natural since a divergence-free vector field is constant if and only if its curl vanishes. Note that $\nabla \times \nabla u =0$, hence $\nabla \times v = \nabla \times w$, i.e. we can formulate regularization either on $v$ or on $w$. The most general formulation is given by
\begin{equation}
	\frac{\lambda}2 \Vert u - f \Vert_2^2 +  \Vert \nabla u - w \Vert_{{\cal M}(\Omega)}
	+ F(\nabla \cdot w) + G(\nabla \times w) \rightarrow \min_{u \in L^2(\Omega),w \in {\cal D}'(\Omega)^d }\,, \label{eq:model2allgemein}
\end{equation}
with convex functionals $F$ and $G$, which however exceeds the scope of this paper and is left as subject of future research. 

We finally recast the above results in terms of the regularization they induce on $u$. In this respect we also discuss the boundary conditions in (\ref{eq:inpainting}), respectively its weak formulation. Natural boundary conditions are no-flux conditions $(\nabla u - v)\cdot n = 0$ on $\partial \Omega$, which means the used weak formulation of (\ref{eq:inpainting}) is
\begin{equation}
	\int_\Omega \Delta \varphi ~u~dx + \int_\Omega \nabla \varphi \cdot dv(x) = 0 \qquad \forall \varphi \in C^2(\Omega), \nabla \varphi \cdot n = 0 \mbox{~on~}\partial \Omega\,. \label{eq:inpaintingweak1}
\end{equation}
Hence, we can define the regularization functional $R:L^1(\Omega) \rightarrow [0,\infty]$
\begin{equation} \label{eq:Rdefinition}
	R(u) := \inf_{v \text{~satisfying~} (\ref{eq:inpaintingweak1})} \Vert v \Vert_{{\cal M}(\Omega)}\,.
\end{equation}

Once we have defined the regularization terms it is straight-forward to extend the variational model to other imaging tasks, e.g.  by just changing the data fidelity. Moreover, we can consider Bregman iterations (cf. \cite{Osher_et_al})
\begin{equation}
	u^{k+1} \in \mbox{arg}\min_u \left( \frac{\lambda}2 \Vert u - f \Vert_2^2 + R(u) - \langle p^k, u \rangle \right)\,, \quad p^k \in R(u^k)\,,
\end{equation}
as well as other scale space methods such as the gradient flow (cf. \cite{Andreu_et_al}) $\partial_t u \in - \partial R(u)$
and the inverse scale space method (cf. \cite{Burger_et_al}).

\section{Properties of Regularization by Sparse Vector Fields}

In the following we further discuss some properties of the regularization functional $R$ defined via (\ref{eq:Rdefinition}). To avoid obvious technicalities with constants, we restrict ourselves to the space
\begin{equation}
L^1_\diamond(\Omega) = \{~u \in L^1(\Omega)~|~\int_\Omega u~dx = 0~\}
\end{equation}
if $\Omega$ is a bounded domain. 

We start with some topological properties induced by $R$:
\begin{theorem}
Let $\Omega$ be a sufficiently regular domain. Then there exists a constant $c > 0$ such that 
\begin{equation}
	\Vert u \Vert_{L^1(\Omega)} \leq c R (u)
\end{equation}
for all $u \in L^1_\diamond(\Omega)$. Moreover, dom$(R)$ is a subspace of $L^1(\Omega)$ and $R$ is a norm on dom$(R) \cap L^1_\diamond(\Omega) $. Finally, 
\begin{equation}
	R(u) \leq TV(u) \qquad \forall~u \in BV(\Omega)\,. \label{RTVestimate}
\end{equation}
\end{theorem}
\begin{proof}
We have
\begin{equation}
	\Vert u \Vert_{L^1(\Omega)} = \sup_{\phi \in L^\infty(\Omega), \Vert \phi \Vert_\infty \leq 1} \int_\Omega u(x) \phi(x)~dx\,.
\end{equation}
For $\phi \in L^\infty(\Omega)$ we define $w$ as the weak solution of the Poisson equation $-\Delta w = \phi$ with homogeneous Neumann boundary conditions and mean value zero. Thus, using the weak formulations we have
\begin{equation}
\int_\Omega u(x) \phi(x)~dx = \int_\Omega \nabla w(x) \cdot~dv(x)\,.
\end{equation}
Regularity of solutions of the Poisson equation yields continuity of $w$ and the existence of a constant $c$ such that
$ \Vert \nabla w \Vert_\infty \leq c \Vert \phi \Vert_\infty = c, $
for all $\phi \in L^\infty(\Omega)$. Hence,
\begin{equation}
\sup_{\phi \in L^\infty(\Omega), \Vert \phi \Vert_\infty \leq 1} \int_\Omega u(x) \phi(x)~dx \leq c \Vert v \Vert_{{\cal M}(\Omega)}\,,
\end{equation}
which yields the estimate of the $L^1$-norm.

The one-homogeneity and triangle inequality follow in a straight-forward way from the definition, consequently $R$ is a norm on a subspace of $L^1_\diamond(\Omega) $.
Estimate (\ref{RTVestimate}) is obtained since $v=\nabla u$ satisfies (\ref{eq:inpaintingweak1}), hence the infimum over all admissible $v$ is less or equal to the total variation.
\end{proof}

A next step towards the understanding of properties of $R$ is an investigation of its subdifferential, with subsequent consequences for optimality conditions of (\ref{eq:model1a}). For brevity we use a formal approach based on Lagrange multipliers. We have $p \in \partial R(u)$ if and only if
$ p = \partial_u L(u,v,q), $
for solutions $(v,q)$ of the saddle-point problem
\begin{equation}
\inf_v \sup_q L(u,v,q)
\end{equation}
for given $u$ and the Lagrangian is defined as
\begin{equation}
L(u,v,q) = \Vert v \Vert_{{\cal M}(\Omega)}+\int_\Omega \Delta q ~u~dx + \int_\Omega \nabla q \cdot dv(x)\,.
\end{equation}
Thus, we find $p=\Delta q$ and the optimality conditions for the saddle point problem yield (\ref{eq:inpaintingweak1}) and $- \nabla q \in \partial \Vert v \Vert_{{\cal M}(\Omega)}. $

It is instructive to compare the subgradients of $R$ with those of TV. Indeed with similar reasoning one can show that $p \in \partial TV(u)$ if $p = - \nabla \cdot g$ for a vector field $g \in \partial \Vert v \Vert_{{\cal M}(\Omega)}$ and $v = \nabla u$. This means that if we can write $g = - \nabla q$ we also obtain $p \in \partial R(u)$. In particular this opens the door towards a simple verification whether solutions of the ROF-model are also solutions of the sparse vector field model (\ref{eq:model1a}). One simply has to inspect the subgradient in the optimality condition and check whether the associated vector field $g$ can be written as a gradient. We will exemplify this in the case of the most well-known example for the ROF model, the reconstruction of the indicator function of a ball on $\Omega = \mathbb{R}^d$ (cf. \cite{meyer}). This function is an eigenfunction of TV, i.e. there exists $\lambda$ (depending on the radius $R$ of the ball) such that
\begin{equation}
	\lambda u = \nabla \cdot g \in \partial TV(u)\,.
\end{equation}
It is easy to see that $g=\nabla F(b)$, where $b$ is the signed distance function of the ball (cf. \cite{Delfour_and_Zolesio}) and $F$ satisfies
\begin{equation}
F'(z) =\begin{cases}1 & \mbox{if~} z \leq 0 \\ \frac{R}{z+R} & \mbox{if~} z > 0\,. \end{cases}
\end{equation}
Hence, we have $g = -\nabla q$ for $q=-F(b)$, which implies that $\lambda u \in \partial R(u)$,
with the same value. The results in \cite{Benning_Burger_2013} imply that the variational model (\ref{eq:model1a}) reconstructs data $f$ being the indicator function of a ball in the form $u = c f$, with $c < 1$ depending on $\lambda$ and $\alpha$. Moreover, the Bregman iteration and inverse scale space methods reconstruct $f$ exactly after a finite number of iterations respectively finite time.

Finally we return to the original idea of compressing an image by encoding a sparse vector field. For this sake it is desireable that $v$ has support on a set of small (or even zero) Lebesgue measure. Thinking about the continuum case as a limit of discrete pixel images, the asymptotic property of zero Lebesgue measure means that the image (in 2D) can be encoded by a number of values proportional to the square root of the number of pixels. Consequently the compression rate of such a PDE-based approach should improve with higher image resolution, which is highly relevant given the current trend of screen and camera resolution.
We already see from the example of the indicator function of a ball above that we can expect the method to encode a piecewise constant image by vector fields concentrated on the edge sets. For more complicated images the vector field potentially needs to have a larger support to obtain a suitable reconstruction, since away from the support of $v$ the function $u$ is just harmonic. A better understanding of the compression properties would need a characterization of the structure of minimizers, similar to \cite{chambollecaselles,valkonen}. While the one-dimensional case is equivalent to total variation regularization and hence always yields $v$ concentrated on a set of zero Lebesgue measure (cf. \cite{Ring}), the multi-dimensional case is less clear and left as an interesting topic of future research. 
Further studies on the compression properties will be carried out below by computational experiments.

\section{Numerical Solution via Primal-Dual Methods}

In order to solve our minimization problem (\ref{eq:model1a}) numerically, we at first need to discretize it. Thereto we will adopt the notation of the continuous functions $u$, $v$, $f$ and the operators $\nabla$, $\nabla\cdot$ and $\Delta$, but from now on, we are thereby referring to their discretized versions, which we shall comment on in the following.

For simplification, we assume the normalized images to be quadratic, i.e. $f,\ u \in [0,1]^{N \times N}$. The pixel grid can be written as $\{(ih,jh)\colon 1 \leq i,j \leq N\}$, where $h$ denotes the spacing size.
We use forward finite differences with Neumann boundary conditions for the discretization of the gradient of $u$ and in order to preserve the adjoint structure the divergence is discretized with backward finite differences.
%
Moreover, in this discrete setting the $d$-dimensional Radon measure on $\Omega \subset \mathbb{R}^d$ becomes the discrete L1-norm. Considering $v$ as being related to the gradient of $u$, the regularization term in problem \eqref{eq:model1a} can be interpreted as the discrete total variation norm. We decided to use its isotropic version which is given by
\begin{equation}
	\Vert \nabla u \Vert_1 = \sum\limits_{i,j} \vert(\nabla u)_{i,j}\vert = \sum\limits_{i,j} \sqrt{((\nabla u)^1_{i,j})^2 + ((\nabla u)^2_{i,j})^2}\,.
\end{equation}
Consequently, the minimization problem \eqref{eq:model1a} reads:
\begin{equation}
	\frac{\lambda}2 \Vert u - f \Vert_2^2 + \Vert v \Vert_1 \rightarrow \min_{u,v} \qquad\mbox{subject to}\quad \Delta u = \nabla \cdot v  \label{eq:discrete_model1}
\end{equation}
or equivalently:
\begin{equation}
	\frac{\lambda}2 \Vert u - f \Vert_2^2 + \Vert v \Vert_1 + \chi_0(\Delta u - \nabla \cdot v) \rightarrow \min_{u,v}\,, \label{eq:discrete_model2}
\end{equation}
where $\chi_0$ is again the characteristic function of the set $\{0\}$. 
Defining
\begin{equation}
x :=\ (u,v)^T\,, \hfill
\ G(x) :=\ \frac{\lambda}2 \Vert u - f \Vert_2^2 + \Vert v \Vert_1\,, \hfill
\ F(Kx) :=\ \chi_0(\Delta u - \nabla \cdot v)\,,
\end{equation}
one can easily see that we can calculate a solution of the above problem \eqref{eq:discrete_model2} by applying a version of the recently very popular primal-dual algorithms (cf. \cite{{Esser_et_al},{Chambolle_and_Pock}}) designed for efficiently solving general minimization problems of the form  
\begin{equation}
F(Kx) + G(x) \rightarrow \min_{x}\,, \label{eq:general_minimization_problem}
\end{equation}
where F and G are proper convex lower-semicontinuous functionals.\\
\begin{algorithm}[h!]
\caption{Primal-Dual Algorithm by Chambolle and Pock}
\label{alg:cp}
{
\begin{algorithmic}
\Require $\tau, \sigma > 0$, $\theta \in [0,1]$
\Ensure $x^0$, $y^0$, $\bar{x}^0=x^0$
\For{$n \geq 0$}
	\State $y^{n+1} = \left( I + \sigma \partial F^* \right)^{-1} \left( y^n + \sigma K \bar{x}^n \right)$
	\State $x^{n+1} = \left( I + \tau \partial G \right)^{-1} \left( x^n - \tau K^* y^{n+1} \right)$
	\State $\bar{x}^{n+1} = x^{n+1} + \theta \left( x^{n+1} - x^n \right)$
\EndFor
\end{algorithmic}
}
\end{algorithm}
We decided to use the first-order primal-dual algorithm as proposed by Chambolle and Pock (cf. \cite{Chambolle_and_Pock}) given by Algorithm \ref{alg:cp}. Adopting their notation we can now derive the updates concerning our minimization problem \eqref{eq:discrete_model2}. Thereto we at first have to calculate the dual functional $F^*(y) = \sup_x \lbrace \langle y,x \rangle - F(x)\rbrace$. Since in our case $F$ is the characteristic function of the set $\lbrace 0 \rbrace$, it is straightforward to see that $F^*$ equals zero. Hence the update for the dual variable $y$ simplifies to $y^{n+1} = y^n + \sigma K \bar{x}^n$.
Next we consider the update for the primal variable $x$. As the subdifferentials of $G$ with respect to $u$ and $v$ are independent of $v$ and $u$, respectively, we can update each component of $x$ separately (see Algorithm 2). Using the norm of the sum of the primal and dual residual given by (cf. \cite{Goldstein_et_al})
\begin{equation}
\Vert P^{n+1} \Vert^2 + \Vert D^{n+1} \Vert^2,
\end{equation}
where
\begin{equation}
P^{n+1} = \frac{x^n - x^{n+1}}{\tau} - K^*(y^n - y^{n+1})\,, \hfill 
\ D^{n+1} = \frac{y^n - y^{n+1}}{\sigma} - K(x^n - x^{n+1})\,, 
\end{equation}
as a stopping criterion the implementation of our minimization problem \eqref{eq:discrete_model2} can be summarized by Algorithm \ref{alg:model1}, where the two-dimensional isotropic shrinkage operator is defined by:
\begin{equation}
shrinkage(z, \gamma) = \max (\Vert z \Vert_1 - \gamma)\frac{z}{\Vert z \Vert_1}\,.
\end{equation}
\begin{algorithm}[h!]
\caption{Primal-Dual Algorithm for Minimization of (\ref{eq:discrete_model2})}
\label{alg:model1}
{
\begin{algorithmic}
\Require image $f$, $\lambda > 0$, $\tau, \sigma > 0$, $\theta \in [0,1]$, max no of iterations, $\epsilon > 0$
\Ensure $u^0$, $v^0$, $y^0$, $\bar{u}^0=u^0$, $\bar{v}^0=v^0$
	\While{primal-dual residual $> \epsilon$ \textbf{and} $n <$ max no of iterations}
		\State $y^{n+1} = \left( I + \sigma \partial F^* \right)^{-1} \left( y^n + \sigma \nabla\cdot\left( \nabla \bar{u}^n - \bar{v}^n \right) \right) = y^n + \sigma \nabla\cdot\left( \nabla \bar{u}^n - \bar{v}^n \right)$
		\State $u^{n+1} = \left( I + \tau \partial_u G \right)^{-1} \left( u^n - \tau \Delta y^{n+1} \right) = \frac{1}{1+\lambda\tau}\left( \lambda\tau f + u^n - \tau \Delta y^{n+1} \right)$
		\State $v^{n+1} = \left( I + \tau \partial_d G \right)^{-1} \left( v^n - \tau \nabla y^{n+1} \right) =$ shrinkage($v^n - \tau \nabla y^{n+1},\tau$)
		\State $\bar{u}^{n+1} = u^{n+1} + \theta \left( u^{n+1} - u^n \right)$
		\State $\bar{v}^{n+1} = v^{n+1} + \theta \left( v^{n+1} - v^n \right)$
	\EndWhile
\end{algorithmic}
}
\end{algorithm}

As already mentioned in Section \ref{chap:Variational_Model} the model discussed so far can be further extended by considering for example Bregman iterations as proposed by Osher and coworkers \cite{Osher_et_al}. To incorporate this iterative regularization method in our algorithm, we use their "adding-back-the-noise" formulation such that the update for $u$ in the previously introduced routine is replaced by 
\begin{equation}
u^{n+1} = \frac{1}{1 + \lambda\tau}(\lambda\tau f + h^k + u^n - \tau \Delta y^{n+1})\,.
\end{equation}
Besides, the existing implementation is extended by an outer loop over a given number of Bregman iterations in which $h$ is updated by $h^{k+1} = h^k + f - u $
after each complete cycle of the inner loop.

\section{Computational Results}

In the following we present some results for the cases of image compression and denoising discussed above. As an example we chose the frequently used image "Trui" ($257 \times 257$ pixels), making the approach comparable to previous results such as \cite{Mainberger_et_al}. However, since the size of this image does not correspond to modern HD resolutions,  we also created two similar images with sizes of $1024 \times 1024$ and $4800 \times 4800$ pixels, respectively. 

\subsection{Image Compression}
\begin{figure}[h]
\centering
\includegraphics[height=5.5cm]{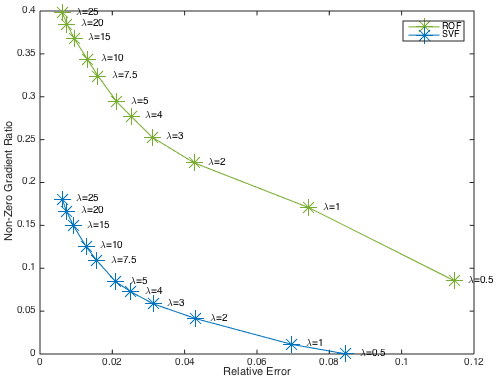} \hspace{0.5cm}
\includegraphics[height=5.5cm]{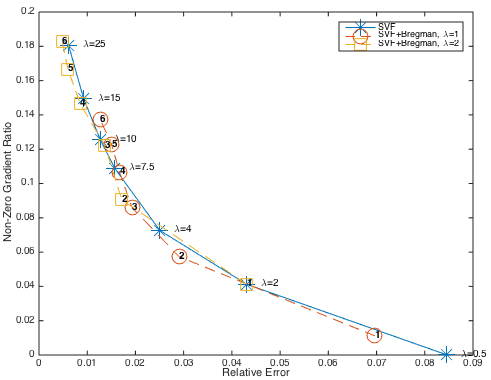}
\caption{Comparison of the non-zero gradient ratio for different parameter values}
\label{fig:ComparisonOfParameterValues}
\end{figure}

We start by discussing the compression of the cartoon part of an image, which is illustrated in Figure \ref{fig:ComparisonOfParameterValues} for the Trui test image. We plot the relative error vs. the non-zero gradient ratio, which means the number of pixels with non-zero $v$ divided by the total number of pixels. On the left we show a comparison of the sparse vector field (SVF) model with the classical ROF model, which demonstrates the improved compression properties. On the right we plot the results for the variational SVF model compared to the Bregman iteration (numbers correspond to Bregman iterations), which illustrates that no significant improvement can be obtained with respect to compression by the latter. In Figure \ref{fig:PrincipleOfImageCompressionAlgorithm} we display the results of the compression and the corresponding vector fields for $\lambda = 10$. One observes that the support of $v$ corresponds well to an edge indicator, confirming the relation to the approach in \cite{Mainberger_et_al}. The reconstructed image seems to preserve the main edges well, but does not have the strict piecewise constant behaviour as total variation regularization, which seems attractive for further reconstruction tasks. 

\begin{figure}[h]
\centering
\includegraphics[height = 2.65cm]{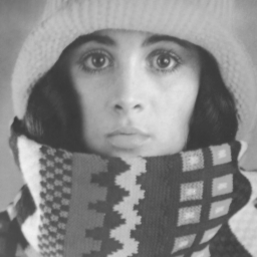}\hfill
\includegraphics[height = 2.65cm]{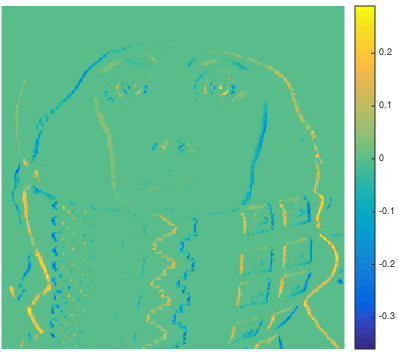}\hfill
\includegraphics[height = 2.65cm]{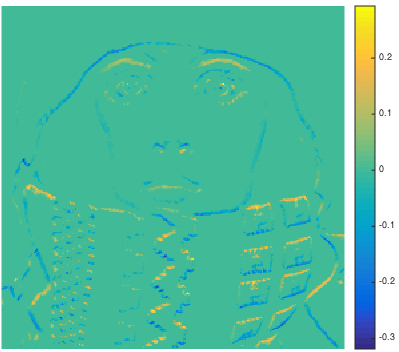}\hfill
\includegraphics[height = 2.65cm]{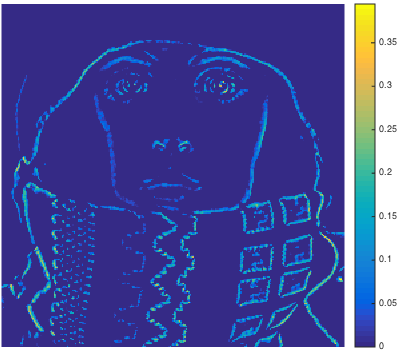}\hfill
\includegraphics[height = 2.65cm]{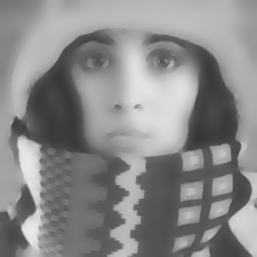}
\caption{From left to right: Original image, vector field $v$ in x- and y-direction, norm of $v$, and the corresponding reconstruction for $\lambda = 10$.}
\label{fig:PrincipleOfImageCompressionAlgorithm}
\end{figure}

We also investigate the behaviour for higher resolution. In order to mimic increasing resolution we simply downscale the test images to $r$ times the number of original pixels, $r \in (0,1]$. We then perform compression at fixed error tolerance (corresponding to constant $\lambda$ when appropriately scaled) for the images of different size and finally plot the non-zero gradient ratio vs. $r$ in Figure \ref{ComparisonOfNon-ZeroGradientRatios}. Our expectation that due to continuum limit and the potential convergence towards a concentrated measure the ratio decreases with increasing resolution is well-confirmed for the Trui image as well as for a similar image at higher resolution. 
\begin{figure}[h]
\centering
\includegraphics[height = 4.25cm]{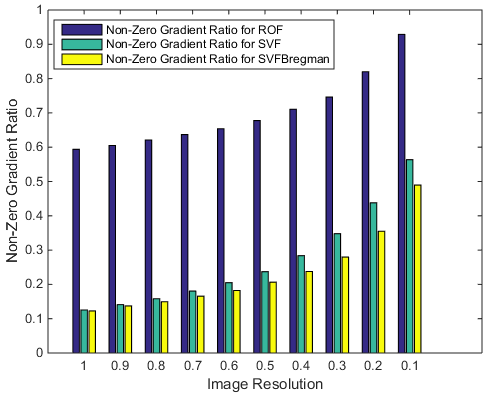}\hspace{0.2cm}
\includegraphics[height = 4.25cm]{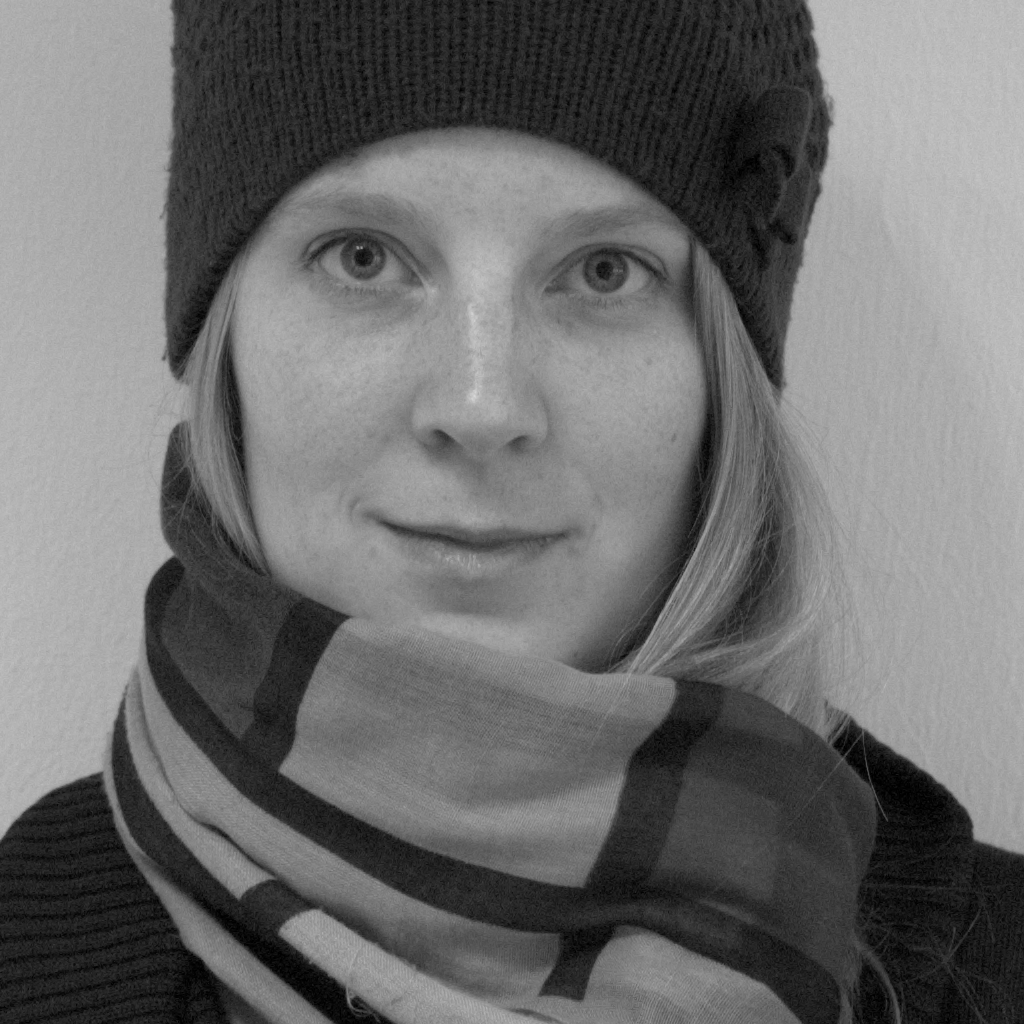}\hspace{0.2cm}
\includegraphics[height = 4.25cm]{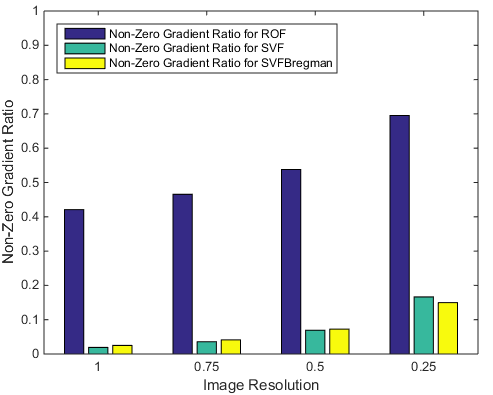}\hspace{0.5cm}
\caption{Comparison of the Non-Zero Gradient Ratios of the SVF ($\lambda = 10)$, the SVFBregman ($\lambda = 1$ and $5$ Bregman Iterations) and the ROF ($\lambda = 10$) algorithms for the Trui test image (left) and an additional test image of size 1024x1024 pixels (right). The latter test image is displayed in the middle.}
\label{ComparisonOfNon-ZeroGradientRatios}
\end{figure}

Finally we display the result of the SVF model for image compression performed on a high definition image and compare it to a jpg image with the same compression rate in Figure \ref{fig:ComparisonWithJPG_HighResolutionTestJoana}. Indeed we achieve an improved PSNR with the SVF model. We also mention that several further compression steps on $v$ can be carried out in an analogous way to \cite{Mainberger_et_al}, which will lead to highly improved rates, but is beyond the scope of this paper. 
\begin{figure}[h]
\begin{minipage}[c]{0.275\textwidth}
\centering
\includegraphics[height = 4.25cm]{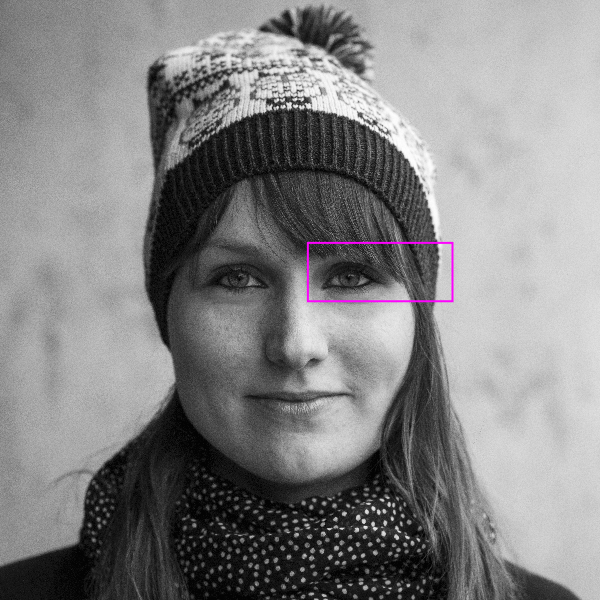}
\end{minipage}\hfill
\begin{minipage}[c]{0.3\textwidth}
\includegraphics[width = 4.75cm]{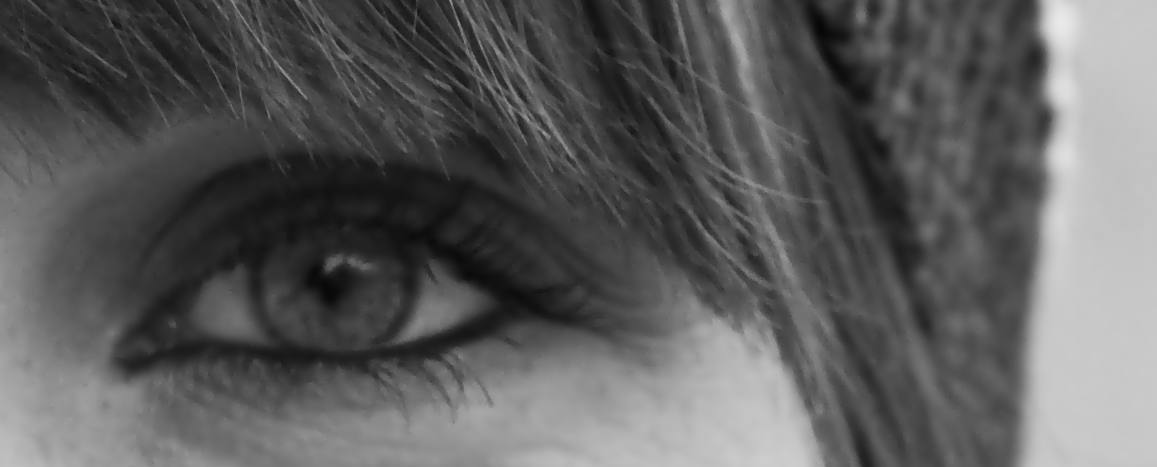}\\
\vfill
\includegraphics[width = 4.75cm]{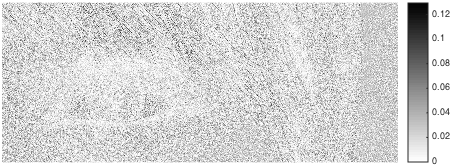}
\end{minipage}\hfill
\begin{minipage}[c]{0.3\textwidth}
\includegraphics[width = 4.75cm]{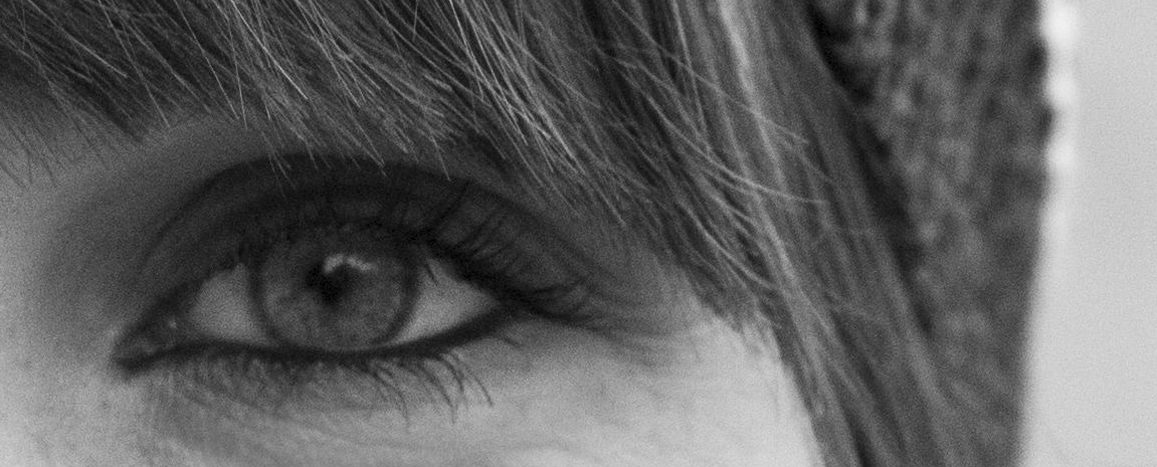}\\
\vfill
\includegraphics[width = 4.75cm]{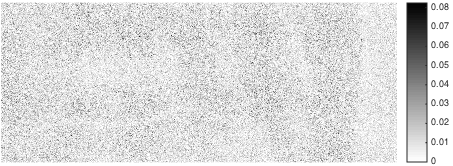}
\end{minipage}
\caption{From left to right: Original image with a resolution of 16.1068 bits per pixel, our reconstruction at a compression of 1.1892 bits per pixel ($\lambda = 10$) with a PSNR value of \textit{34.0767 dB}  and the jpg image at the same compression rate with a PSNR value of 33.3214 dB, corresponding difference images to original one in bottom row.}
\label{fig:ComparisonWithJPG_HighResolutionTestJoana}
\end{figure}

\subsection{Denoising}

We finally give an outlook towards other tasks such as denoising with sparse vector fields. For this sake we compare the model with results of the classical ROF model and choose in both cases $\lambda$ such that the PSNR to the original image is maximized. We illustrate the result in Figure \ref{fig:DenoisingTrui}, which appears to be representative for all our tests. One observes that the reduced staircasing in the SVF model compared to the ROF model is less visible, which is due to point like artefacts that were not present without noise. This results in a lower PSNR than for the ROF model, which is consistent for all our tests. The reason for the artefacts is that $v$ is too sparse in this case and does not encode the contours anymore. This can be seen in the last image of Figure \ref{fig:DenoisingTrui}.
\begin{figure}[t]
\centering 
\includegraphics[height=3.5cm]{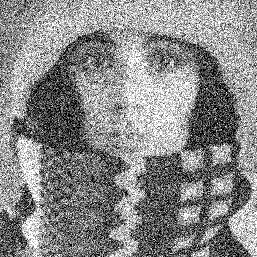}\hspace{0.15cm}
\includegraphics[height=3.5cm]{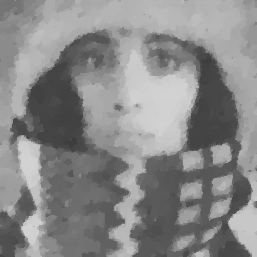}\hspace{0.15cm}
\includegraphics[height=3.5cm]{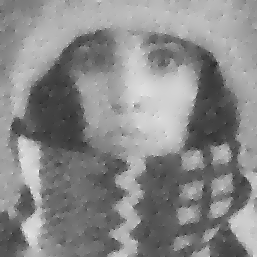}\hspace{0.15cm}
\includegraphics[height=3.5cm]{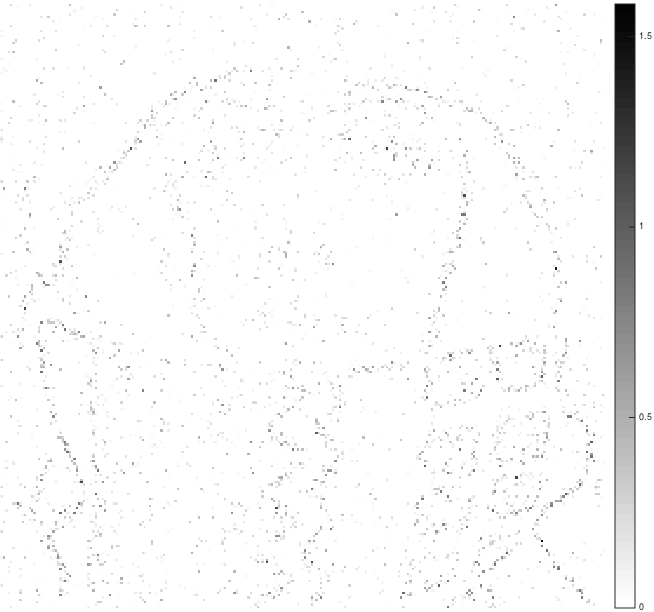}
\caption{Noisy image (Gaussian noise, variance 0.05, left), ROF denoising result ($\lambda=5$, middle left), SVF denoising result ($\lambda=4$, middle right), norm of vector field $v$ in SVF model (right).}
\label{fig:DenoisingTrui}
\end{figure}

\section{Conclusion}

We have introduced the SVF model for image compression motivated by diffusion inpainting and found several interesting connections to TV-type regularization methods. The SVF approach leads to significantly sparser vector fields than the gradients of total variation, which appears quite useful for compression tasks, and seems not to suffer from staircasing artefacts, which appears attractive for other reconstruction tasks. However, the denoising performance of the SVF model is not convincing, since it creates point artefacts at reasonable choice of the regularization parameter. This is probably due to the fact that the norm induced by the corresponding regularization is too weak (there is an upper but no lower bound in terms of TV). For future improvement it seems natural to consider regularization on the curl of the vector field as well, such that $v$ becomes again concentrated on contours rather than scattered points. In particular we suggest to study the more general problem 
\begin{equation}
	\frac{\lambda}2 \Vert u - f \Vert_2^2 + \Vert \nabla u - w \Vert_1 + \beta \Vert \nabla \times w \Vert_1 + \gamma \Vert \nabla \cdot w \Vert_1 \rightarrow  \min_{u,w}  \label{eq:model2}
\end{equation}
for positive $\beta$ and $\gamma$ where again $w= \nabla u - v$ (see \eqref{eq:model2allgemein}). Due to the exact penalization properties of one-norms we expect that the ROF model corresponds to the case of $\beta$ and $\gamma$ sufficiently large, while the SVF model in this paper is $\beta = 0$ and $\gamma$ sufficiently large.

\section*{Acknowledgments} This work has been supported by ERC via Grant EU FP 7 - ERC Consolidator Grant 615216 LifeInverse. MB acknowledges support by the German Science Foundation DFG via  EXC 1003 Cells in Motion Cluster of Excellence, M\"unster, Germany. JG acknowledges support by the Cambridge Biomedical Research Centre. The authors thank Tim L\"opmeier (M\"unster) and Johannes Hjorth (Cambridge) for the acquisition of HD photographs used as test data.


\begin{thebibliography}{4}

%
%
%
%
%

\bibitem{Andreu_et_al} Andreu, F., Ballester, C., Caselles, V., Mazon, J.M.: Minimizing total variation flow. Differential and integral equations 14: 321--360 (2001)


\bibitem{Benning_highOrderTV_2013}
Benning, M., Brune, C., Burger, M., M{\"u}ller, J.: Higher-order {TV}
  methods: {E}nhancement via {B}regman iteration. J Sci Comput  54,  269--310
  (2013)

\bibitem{Benning_Burger_2013}
Benning, M., Burger, M.: Ground states and singular vectors of convex
  variational regularization methods. Methods and Applications of Analysis
  20:  295--334 (2013)
		
\bibitem{Bredies_et_al} Bredies, K., Kunisch, K., Pock, T.: Total generalized variation. SIAM J. Imaging Sci. 3: 492--526 (2010)

\bibitem{Bredies_Pikkarainen} Bredies, K., Pikkarainen, H.K.: Inverse problems in spaces of measures. ESAIM: Control, Optimisation and Calculus of Variations, 19: 190--218 (2013)

\bibitem{Burger_et_al} Burger, M., Gilboa, G., Osher, S., Xu, J.: Nonlinear inverse scale space methods. Communications in Mathematical Sciences 4: 179--212 (2006)


\bibitem{chambollecaselles}Chambolle, A., Caselles, V., Cremers, D., Novaga, M., Pock, T. . An introduction to total variation for image analysis. In: Fornasier, M. (ed.), Theoretical foundations and numerical methods for sparse recovery, DeGruyter, Berlin, 263-340 (2010).

\bibitem{Chambolle_and_Pock} Chambolle, A., Pock, T.: A first-order primal-dual algorithm for convex problems with applications to imaging. J. Math. Imaging Vis., 40: 120--145 (2011)

\bibitem{Chan_et_al} Chan, T., Marquina, A., Mulet, P.: High-order total variation-based image restoration. SIAM Journal on Scientific Computing 22:503--516 (2000)

\bibitem{Delfour_and_Zolesio} Delfour, M.C., Zolesio, J-P.: Shapes and geometries: metrics, analysis, differential calculus, and optimization. SIAM, Philadelphia (2011)

\bibitem{Esser_et_al} Esser, E., Zhang, X., Chan, T.: A general framework for a class of first order primal-dual algorithms for convex optimization in imaging science. SIAM J. Imaging Sci., 3(4): 1015--1046 (2010)

\bibitem{Goldstein_et_al} Goldstein, T., Esser, E., Baraniuk, R.: Adaptive Primal-Dual Hybrid Gradient Methods for Saddle-Point Problems. arxiv:1305.0546v1 [math.NA] (preprint) (2013)

\bibitem{Mainberger_et_al} Mainberger, M., Bruhn, A., Weickert, J., Forchhammer, S.: Edge-based compression of cartoon-like images with homogeneous diffusion. Pattern Recognition 44.9: 1859--1873 (2011)

\bibitem{Mainberger_and_Weickert} Mainberger, M., Weickert, J.: Edge-Based Image Compression with Homogeneous Diffusion. In: Computer Analysis of Images and Patterns. Ed. by Jiang, X., Petkov, N., pp. 476--483. Springer, Berlin, Heidelberg (2009)

\bibitem{meyer}Meyer, Y.: Oscillating Patterns in Image Processing and Nonlinear Evolution Equations, AMS, Providence (2001).

\bibitem{Osher_et_al} Osher, S., Burger, M., Goldfarb, D., Xu, J., Yin, W.: An Iterative Regularization Method for Total Variation-Based Image Restoration. Multiscale Model. Simul., 4: 460--489 (2005)

\bibitem{Ring} Ring, W.: Structural properties of solutions to total variation regularization problems. ESAIM: Math. Modelling  Numer. Analysis 34: 799-810 (2000)

\bibitem{Rudin_et_al} Rudin, L.I., Osher, S., Fatemi, E.: Nonlinear total variation based noise removal algorithms. Physica D 60: 259--268 (1992)

\bibitem{valkonen}Valkonen, T., The jump set under geometric regularisation. Part 1: Basic technique and first-order denoising. arXiv preprint arXiv:1407.1531 (2014).


\bibitem{Weickert_et_al} Weickert, J., Hagenburg, K., Breuss, M., Vogel, O.: Linear Osmosis Models for Visual Computing. Energy Minimization Methods in Computer Vision and Pattern Recognition. Springer Berlin Heidelberg (2013)


\end{thebibliography}
\end{document}